\documentclass[psamsfonts]{amsart}

\usepackage{amssymb,amsfonts}
\usepackage[all,arc]{xy}
\usepackage{enumerate}
\usepackage{enumitem}
\usepackage{mathrsfs}
\usepackage{xcolor}

\newtheorem{thm}{Theorem}[section]
\newtheorem{cor}[thm]{Corollary}
\newtheorem{prop}[thm]{Proposition}

\theoremstyle{definition}

\newcommand{\F}{\mathbb F}
\newcommand{\K}{\mathbb K}

\newcommand{\cD}{\mathcal D}

\newcommand{\cS}{\mathcal S}
\newcommand{\cC}{\mathcal C}

\newcommand{\Tr}{\mathrm{Tr}}

\newcommand{\Aut}{\mathrm{Aut}}

\newcommand{\End}{\mathrm{End}}

\def\F{\mathbb F}
\def\cA{\mathcal A}

\def\cT{\mathcal T}
\def\V{\mathbb V}
\def\cG{\mathcal G}

\def\cU{\mathcal U }
\def \cH {\mathcal H}

\def\zhou#1 {\fbox {\footnote {\ }}\ \footnotetext { From Yue: {\color{red}#1}}}
\def\trombetti#1 {\fbox {\footnote {\ }}\ \footnotetext { From Rocco: {\color{blue}#1}}}

\begin{document}

\title{Automorphism groups and new constructions of maximum additive rank metric codes with restrictions}

\author{G. Longobardi, G. Lunardon, R. Trombetti, Y. Zhou}

\begin{abstract}
Let   $d, n \in \mathbb{Z}^+$ such that  $1\leq d \leq n$. A {\it $d$-code} $\cC \subset \F_q^{n \times n}$ is a subset of order $n$ square matrices with the property that for all pairs of distinct elements in $\cC$, the rank of their difference  is greater than or equal to $d$. A $d$-code with as many as possible elements is called a {\it maximum} $d$-code. The integer $d$ is also called the {\it minimum distance} of the code. When $d<n$, a classical example of such an object is the so-called {\it generalized Gabidulin code}, \cite{kshevetskiy_new_2005}.
In \cite{delsarte_alternating_1975}, \cite{schmidt2} and \cite{schmidt_hermitian}, several classes  of maximum $d$-codes made up respectively of symmetric, alternating and hermitian matrices were  exhibited.  In this article we focus on such examples. 

\noindent Precisely, we  determine their automorphism groups and solve the equivalence issue for them.  Finally,  we exhibit a maximum symmetric $2$-code which is not equivalent to the one with same parameters constructed in   \cite{schmidt2}.
\end{abstract}

\maketitle

\section{Introduction}\label{sec:intro}

Let $\F_q$ be the finite field with $q$ elements and denote by $\F_q^{n \times n}$ the set of order $n$ matrices with entries in $\F_q$. It is easy to verify that the map $d$ defined by 
$$d(A,B) = \mathrm{rank} (A-B),$$ 
for $A,B \in \F_q^{n\times n}$, is a distance function on $\F_q^{n \times n}$, which is often called the {\it rank distance} or the {\it rank metric} on $\F_q^{n \times n}$. 

Given any integer $1\leq d \leq n$, we consider here subsets $\cC \subset \F_q^{n \times n}$ with the property that, for all distinct matrices  $M_1$ and $M_2 \in \cC,$ the rank of $M_1-M_2$ is greater than or equal to $d$. These sets are usually called {\it rank metric codes} with minimum distance $d$, and in some context also $d$-codes. Also, we say that a $d$-code $\cC \subset \F_q^{n \times n}$ is {\it additive} if $\cC$ is a subgroup of $(\F_q^{n \times n},+)$. An {\it $\F_q$-linear} $d$-code is a subspace of  $\F_q^{n \times n}$ viewed as an $n^2$-dimensional vector space over $\F_q$.

For the applications in classical coding theory, given $n$ and $d$, it is desirable to have $d$-codes which are maximum in size. In the general case, which means if it is not required that all elements in the set must possess specific restrictions, Delsarte proved that this bound is $q^{n(n-d+1)}$ (the so-called {\it Singleton-like bound} for rank distance codes) \cite{delsarte_bilinear_1978}. If the cardinality of the code $\cC$ meets this bound, we say that $\cC$ is a {\it Maximum Rank Distance code}, ({\it MRD-code}, for short), or {\it maximum d-codes}. 

Let $\F_{q^n}$ be a finite field of order $q^n$, $q$ a prime power. Let $$\mathcal{L}_{q}[x]=\left\{f(x)=\sum_{i=0}^{k} c_i x^{q^i}: c_i \in \F_{q^n},\, k \in \mathbb{Z}^+ \right\},$$i.e., the set of so-called {\it linearized polynomials} over $\F_{q^n}$ (or {\it $q$-polynomials}). If $k$ is the largest integer such that $c_k \neq 0$, we say that $k$ is the $q$-degree of $f$. 

Rank metric codes consisting of order $n$ square matrices can be considered also in $q$-polynomial representation. 

Indeed it is well known that $\mathcal{L}_{(n,q)}[x]=\mathcal{L}_{q}[x]/(x^{q^n}-x)$ is equivalent to $\mathrm{End}_{\F_q}(\F_{q^n})$, i.e., the set of all endomorphisms of $\F_{q^n}$ seen as a vector space over $\F_q$. Hence, the algebraic structure $(\mathcal{L}_{(n,q)}[x],+,\circ,\cdot)$, where $+$ is addition of maps, $\circ$ is the composition of maps (mod $x^{q^n}-x$)  and $\cdot$ is the scalar multiplication by elements of $\F_q$, is isomorphic to the algebra  $\F_q^{n \times n}$.

Let $\mathrm{Tr}_{q^n / q}$ be the trace function of $\F_{q^n}$ over $\F_q$. The map
\begin{equation}\label{form:bilinearform}
T: (x,y) \in \F_{q^n} \times \F_{q^n}  \rightarrow \mathrm{Tr}_{q^n / q}(xy) \in \F_q,
\end{equation}
 is a non-degenerate $\F_q$-bilinear form of $\F_{q^n}$. 

Let $f(x)=\sum^{n-1}_{i=0} a_i x^{q^{i}}$ be an $\F_q$-linear map of $\F_{q^n}$. Using the terminology of \cite{Sheekey}, we denote by  $f^{\top}$ the {\it adjoint} map of $f$ with respect to $T$; i.e., $$ f^{\top}(x)=\sum_{i=0}^{n-1} a_{n-i}^{q^{i}}x^{q^{i}}.$$ If $f=f^{\top}$, we say that $f$ is \textit{self-adjoint} with respect to the bilinear form defined in (\ref{form:bilinearform}). If $\cC$ is a code consisting of $q$-polynomials, then the adjoint code of $\cC$ is $\cC^{\top}=\{f^{\top} \,:\,  f \in \cC\}$. In fact, the adjoint of $f$ is equivalent to the {\it transpose} of the matrix in $\F_q^{n \times n}$ derived from $f$. 

In the literature, codes in the rank metric context are studied up to several definitions  of equivalence relation; see \cite{de_la_cruz_algebraic_2016,morrison_equivalence_2013}. For what is needed here we may say that two sets of $q$-polynomials over $\F_{q^n}$, say $\cC$ and $\cC'$, are equivalent if there exist two permutation $q$-polynomials $g_1$, $g_2$ and $\rho \in \Aut(\F_q)$ such that  
\begin{equation}\label{eq:generalequivalence} 
	\cC' =\{ g_1\circ f^\rho \circ g_2(x) +h(x) : f\in \cC \},
\end{equation} 
where $h(x) \in\mathcal{L}_{(n,q)}[x]$, and $(\sum a_{i}x^{q^i})^\rho:= \sum a_{i}^\rho x^{q^i}$. Although, in general {\it isometric equivalence} covers the possibility when $$ \cC' =\{ g_1\circ f^{\top \rho} \circ g_2(x) +h(x) : f\in \cC \};$$  see for instance \cite{wan_geometry_1996}.

We indicate the fact that $\cC$ and $\cC'$ are equivalent codes by the symbol $\cC \simeq \cC'$, and denote by $[\cC]_{\simeq}$ the equivalence class of $\cC$ with respect to relevant equivalence relation.  

Let $g_1,\rho,g_2,h$ be as above. In the following we will use the symbol $\Phi_{g_1,\rho,g_2,h}$ to denote the map of $\mathcal{L}_{(n,q)}[x]$ defined by $$f(x) \mapsto g_1 \circ f^{\rho} \circ g_2(x)+h(x).$$  The \emph{automorphism group} of $\cC$ consists of all $\Phi_{g_1,\rho,g_2,h}$ fixing $\cC$.

In this paper we will be mainly interested in the case when the sets $\cC$ and $\cC'$ are additive. It is not difficult to see that if this is the case, we may assume $h(x)$ to be the null map in the definitions above.

If $n=d$, then $\# \cC \leq q^n$. When the equality holds such a set consists of $q^n$ invertible endomorphisms of $\F_{q^n}$. Hence, $\cC$ is a {\it spread set} of $\mathrm{End}_{\F_q}(\F_{q^n})$, and if $\cC$ is additive this is also equivalent to a {\it semifield spread set} of $\mathrm{End}_{\F_q}(\F_{q^n})$. For more results on semifields and related structures, we refer to \cite{johnson_handbook_2007}, \cite{lavrauw_semifields_2011}.

In the case when $d < n$, the most important example of additive $MRD$-code of $\mathcal{L}_{(n,q)}[x]$, is the so-called {\it Generalized Gabidulin code}. This family was found by Kshevetskiy and Gabidulin in \cite{kshevetskiy_new_2005}. It appeared as a  generalization of the family discovered many years before independently by Gabidulin \cite{gabidulin_MRD_1985} and Delsarte \cite{delsarte_bilinear_1978}, whose elements are nowadays known with the name of Delsarte-Gabidulin codes. 

Precisely, let $k,n$ be positive integers and let $s$ be an integer coprime with $n$; a Generalized Gabidulin code with stated parameters is the set of linearized polynomials
\begin{equation}
\cG_{n,k,s}=  \left\{  \sum_{i=0}^{k-1} a_i x^{q^{si}}  \, : \,\, a_0,a_1,\ldots, a_{k-1} \in \F_{q^n}  \right\}.
\end{equation}

The code $\cG_{n,k,s}$ is an $\F_q$-subspace of $\mathcal{L}_{(n,q)}[x]$ of dimension $kn$, hence it has size $q^{nk}$, and any non-zero element in $\cG_{n,k,s}$ has rank greater than or equal to $d=n-k+1$. Hence, $\cG_{n,k,s}$ is an $\F_{q}$-linear MRD-code with minimum distance $d$. 

In \cite{delsarte_alternating_1975}, \cite{schmidt2} and \cite{schmidt_hermitian}, constructions of this sort have been exhibited for sets of linearized polynomials with prescribed restrictions. Precisely, for polynomials associated with symmetric, alternating and hermitian forms. In all such settings a heavy use of the theory of {\it association schemes} led to the determination of bounds on the size of such $d$-codes. Moreover, in the additive case such bounds are proven to be tight by exhibiting families of $\F_q$-linear examples attaining these bounds. 

In this article we elaborate on such maximum $\F_q$-linear codes. Precisely, in Section 3 we determine their automorphisms group and solve the equivalence issue for them.  In Section $4$, we characterize relevant $d$-codes as the intersection of their ambient space with a suitable code which is equivalent to a generalized Gabidulin code with minimum distance $d$. Finally, in Section 5 we exhibit a symmetric $2$-code of order $q^{2m^2}$, which is not equivalent to the one with same parameters constructed in   \cite{schmidt2}.

\section{Preliminaries}

We start this section by giving a description of the known examples of maximum additive $d$-codes presented in \cite{delsarte_bilinear_1978}, \cite{schmidt2} and \cite{schmidt_hermitian}, in terms of $q$-polynomials. 

In order to do that we first remind the following very well known fact, which in the symmetric setting is stated for instance in \cite[Lemma 13]{schmidt1}:

\begin{prop}\label{bilinearform}
	Let $\ell$ be an arbitrary integer.
	\begin{enumerate}
		\item For each $m$-dimensional $\F_{q}$-subspace $U$ of $\F_{q^n}$, every  bilinear form $B: U \times \F_{q^n} \rightarrow \F_q$ can be written in the following form 
		\begin{equation*}
		B(x,y)=\mathrm{Tr}_{q^n / q}   \Biggl ( \sum_{j=0}^{m-1} a_j yx^{q^{(j-\ell)}} \Biggr ),
		\end{equation*}
		for some uniquely determined $a_0, a_1, \ldots,a_{m-1} \in \F_{q^n}$.
		\item For each $m$-dimensional $\F_{q^2}$-subspace $U$ of $\F_{q^{2n}}$, every Hermitian form $H: U \times \F_{q^n} \rightarrow \F_{q^2}$ can be expressed in the form
		\begin{equation*}
		H(x,y)=\mathrm{Tr}_{q^{2n}/q^2}   \Biggl ( \sum_{j=0}^{m-1} a_j y^{q}x^{q^{2(j-\ell)}} \Biggr ),
		\end{equation*}
		for some uniquely determined $a_0, a_1, \ldots,a_{m-1} \in \F_{q^{2n}}$.
	\end{enumerate}
\end{prop}

In particular, each bilinear form say $B(\cdot,\cdot)$ defined over $\F_{q^n}$, seen as a vector space over $\F_q$, can be written in the following shape: $$B(x,y)=\mathrm{Tr}_{q^n /q}(f(x)y),$$ where $f(x) \in \mathcal{L}_{(n,q)}[x]$.

\subsection{Known constructions in the symmetric and alternating setting}

A \textit{symmetric} $\F_q$-bilinear form $B$ of $\F_{q^n}$ is a bilinear form such that for each  $x,y \in \F_{q^n}$,
\begin{equation} \label{symmetry}
B(y,x)=B(x,y). 
\end{equation}
By Proposition \ref{bilinearform}, there is  a $q$-polynomial $f(x)$ such that $B(x,y)=\mathrm{Tr}_{q^n/q}(f(x)y),$ and by (\ref{symmetry}) we must have for all $x,y \in \F_{q^n}$, $\mathrm{Tr}_{q^n/q}(f(y)x)=\mathrm{Tr}_{q^n/q}(f(x)y)$. It is routine to verify that
\begin{equation*}
\mathrm{Tr}_{q^n/q}(f(y)x)=\mathrm{Tr}_{q^n/q}(f(x)y)=\mathrm{Tr}_{q^n/q}(xf^{\top}(y)),
\end{equation*}
which means that $f$ is a self-adjoint map with respect to $T$ given in \eqref{form:bilinearform}.

Therefore, by suitably choosing an $\F_q$-basis of $\F_{q^n}$, we can identify the set of symmetric bilinear forms over $\F_{q^n}$, with the $\frac{n(n+1)}{2}$-dimensional subspace $S_{n}(q) \subset \End_{\F_q}( \F_{q^n})$ of self-adjoint $\F_q$-linear maps of $\F_{q^n}$. Precisely, 
 
\begin{equation}\label{symplectic space}
S_n(q)= \Biggl \{  \sum_{i=0}^{n-1} c_i x^{q^{i}}  \, : \, c_{n-i}=c^{q^{(n-i)}}_i \,\,\, \text{for } i \in \{0,1,\ldots, n-1 \} \Biggr \}.
\end{equation}

An \textit{alternating} $\F_q$-bilinear form $B$ of $\F_{q^n}$ instead is a bilinear form such that for all $x \in \F_{q^n},$

\begin{equation} \label{alternate}
B(x,x)=0;
\end{equation}
from which the additional property 

\begin{equation}\label{alternateproperty}
B(x,y)+B(y,x)=0
\end{equation}
follows.

By Proposition \ref{bilinearform}, Equations (\ref{alternate}) and (\ref{alternateproperty}), and again properly choosing an $\F_q$-basis of $\F_{q^n},$ the set of alternating bilinear form with entries running over $\F_q$ can be seen as the following subset of $q$-polynomials:

\begin{equation} \label{alternatingspace}
A_{n}(q)= \Biggl \{  \sum_{i=1}^{n-1} c_i x^{q^{i}}  \, : \, c_{n-i}=-c^{q^{(n-i)}}_i \,\,\, \text{for } i \in \{1,2,\ldots, n-1 \} \Biggr \}.
\end{equation}
Clearly, $A_{n}(q)$ is an $\frac{n(n-1)}{2}$-dimensional subspace of $\End_{\F_q}(\F_{q^n})$ and it is well known that the \textit{rank} of each element of $A_{n}(q)$, is necessarily even.

Denote by the symbol $X_{n}$ either the subspace $S_{n}(q)$ or $A_{n}(q)$.  It is readily verified that for given $a \in \F^*_{q}$, $\rho \in \Aut(\F_q)$, $g$ a permutation $q$-polynomial over $\F_{q^n}$, and $r_0 \in X_{n}$, the map $\Psi: X_{n} \rightarrow X_{n}$ defined by 
\begin{equation}\label{eq:def_equivalence}
	\Psi_{a,g,\rho,r_0}(f) = ag \circ f^{\rho} \circ g^{\top}(x) + r_0(x),
\end{equation}
preserves the rank distance on $X_{n}$. In fact, the converse statement is also true except when $q=2$ and $n=3$ if $X_{n}=S_{n}(q)$, and except when $n \leq 3$ if $X_{n}=A_{n}(q)$; see \cite{wan_geometry_1996}. 

For two subsets $\cC_1$ and $\cC_2$ of $X_n$, if there exists a map $\Psi_{a,g,\rho,r_0}$ defined as in Equation \eqref{eq:def_equivalence} for certain $a$, $g$, $\rho$ and $r_0$ such that
\[ \cC_2 :=\{\Psi_{a,g,\rho,r_0}(f): f\in \cC_1 \},\]
then we say that $\cC_1$ and $\cC_2$ are \emph{equivalent} in $X_{n}$, and to distinguish this relation from the one defined in Section \ref{sec:intro}, we write $\cC_1 \cong \cC_2$.

Regarding upper bounds for such $d$-codes, parts of the following results can be found in  \cite[Theorem 3.3]{schmidt2} and \cite[Corollary 7, Remark 8]{schmidt1}, and the last open case that $q$ and $d$ both even was proved in \cite{schmidt_quadratic}.
\begin{thm}\cite{schmidt_quadratic}\label{th:upper_bound}
	Let $\cC$ be a $d$-code in $\cS_n(q)$, where $\cC$ is required to be additive if $d$ is even. Then
	\begin{equation}\label{eq:upperbound_general}
		\#\cC \le \left\{
		\begin{array}{ll}
		q^{n(n-d+2)/2}, & \text{if $n-d$ is even;} \\ 
		q^{(n+1)(n-d+1)/2}, & \text{if $n-d$ is odd.}
		\end{array} 
		\right.
	\end{equation}
\end{thm}

Recall that in the alternating setting, the rank of matrices are always even. We have a result of the same sort due to Delsarte and Goethals; precisely,

\begin{thm}\cite{delsarte_alternating_1975}\label{th:upper_bound_alternating_association_scheme}
	Let $m=\lfloor \frac{n}{2} \rfloor$ and assume that $\cC$ is any $2e$-code in $A_n(q)$, then $$\# \cC \leq q^{\frac{n(n-1)}{2m}(m-e+1)}.$$
\end{thm}

Also in \cite{delsarte_alternating_1975}, Delsarte and Goethals exhibited a class of $\F_q$-linear maximal codes in $A_{n}(q)$ for any characteristic,  and any odd value of $n$. 

Precisely, let $ 2 \leq d=2e \leq n-1$, and let $s$ be an integer coprime with $n$.  Then the set of $q$-polynomials
\begin{equation}\label{eq:alternatingcode}
\mathcal{A}_{n,d,s}=  \Biggl \{  \sum_{i=e}^{\frac{n-1}{2}} \biggl ( b_i x^{q^{si}}- (b_i x)^{q^{s(n-i)}} \biggr ) \, : \,\, b_{e},\ldots, b_{\frac{n-1}{2}} \in \F_{q^n}  \Biggr \}
\end{equation}
is a maximum alternating $d$-code \cite[Theorem 7]{delsarte_alternating_1975}.

In \cite{schmidt2}, Kai-Uwe Schmidt presented the following class of additive (in fact, $\F_q$-linear) codes in $S_{n}(q)$.  For any integer $1 \leq d \,\, \leq \,\, n$ such that $n-d$ is even and $s$ coprime with $n$, consider the following subset of  $S_n(q)$:
\begin{equation} \label{Schimdtcode}
\mathcal{S}_{n,d,s} = \Biggl \{ b_0 x+  \sum_{i=1}^{\frac{n-d}{2}}   \Bigl ( b_i x^{q^{si}}+(b_ix)^{q^{s(n-i)}}  \Bigr ) \, : \,\, b_0, b_1, \ldots, b_{\frac{n-d}{2}} \in \F_{q^n}  \Biggr \}.
\end{equation}

The set $\mathcal{S}_{n,d,s}$ turns out to be a  maximum $d$-code \cite[Theorem 4.4]{schmidt2}. Also in \cite{schmidt2} the author showed that for any such  $d$, it is always possible to construct a maximal $d$-code of $S_{n}(q)$ with $n-d$ an odd integer; in fact, by simply \textit{puncturing} the $(d+2)$-code $\cS_{n+1,d+2,s}$ of $S_{n+1}(q)$ \cite[Theorem 4.1]{schmidt2}.

\subsection{Known constructions in the Hermitian setting}

Let $\F_{q^{2n}}$ be the finite field of order $q^{2n}$ equipped with the involuntary automorphism  $a \mapsto a^q$ of $\F_{q^2}$. 

A \textit{Hermitian form} on $\F_{q^{2n}}$, is a map
\begin{equation*}
H:\F_{q^{2n}} \times \F_{q^{2n}} \rightarrow  \F_{q^2}
\end{equation*}
which is $\F_{q^2}$-linear in the first coordinate and satisfies the following property
\begin{equation}\label{hermitiancondition}
 H(y,x)=H(x,y)^q,
\end{equation}
for all $x,y \in \F_{q^{2n}}$. 

It  is easy to check that for all $x \in \F_{q^{2n}}$, $\mathrm{Tr}_{q^{2n}/q^2}(x)^q=\mathrm{Tr}_{q^{2n}/q^2}(x^q)$.  

Also, the map
\begin{equation*}
	S: (x,y) \in \F_{q^{2n}} \times \F_{q^{2n}} \rightarrow \mathrm{Tr}_{q^{2n}/q^2}(x y^q)
\end{equation*}
is a non-degenerate sesquilinear form of $\F_{q^{2n}}$ with companion automorphism  $a \mapsto a^q$.

Again by Proposition \ref{bilinearform} (b), every such a sesquilinear form can be written in the following fashion:
\begin{equation*} \label{hermitiantrace}
H(x,y)=S(f(x),y)=\mathrm{Tr}_{q^{2n}/q^2}(f(x)\, y^q),
\end{equation*}
where $f(x) \in  \mathcal{L}_{(n,q^2)}[x] $ is a $q^2$-polynomial with coefficients in $\F_{q^{2n}}$. 

Now, let $f(x)=\sum^{n-1}_{i=0} a_i x^{q^{2i}}$ be an element of $\mathcal{L}_{(n,q)}[x]$ (which can be viewed as an element of $\End_{\F_{q^2}}(\F_{q^{2n}})$). It is easy to show that $S(f(x),y)^q= S(\tilde{f}(y),x)$ for all $x,y \in  \F_{q^{2n}}$ where
\begin{equation*}
\tilde{f}(x)=f^{\top q}(x^{q^2})=\sum_{i=0}^{n-1} a^{q^{2n-2i+1}}_ix^{q^{2(n-i+1)}}.
\end{equation*}Here $f^\top$ denotes the adjoint map of $f$ as an $\F_{q^2}$-linear map, i.e.,\ $f^\top =\sum_{i=0}^{n-1} a_{n-i}^{q^{2i}}x^{q^{2i}}$. It is routine to verify that $\tilde{(\cdot)}$ is involutionary on each $\F_{q^2}$-linear map.

Then by (\ref{hermitiancondition}), we obtain
\begin{equation*}
S(f(y),x)=H(y,x)=H(x,y)^q=S(f(x),y)^q=S(\tilde{f}(y),x)
\end{equation*}
for all $x,y \in \F_{q^{2n}}$.
 
Hence, we may identify the set of Hermitian forms defined on $\F_{q^{2n}}$  with the set of $q^2$-polynomials

\begin{equation} \label{hermitianspace}
H_n(q^2)=
\Biggl \{\sum_{i=0}^{n-1} c_i x^{q^{2i}}  \, : \,\, c_{n-i+1}=c^{q^{2n-2i+1}}_i, \,\,\,\,\, i\in \{0,1,2,\ldots,n-1\} \Biggr \},
\end{equation}
where the indices of the $c_i$'s are taken modulo $n$. The set $H_n(q^2)$ is an $n^2$-dimensional $\F_q$-vector subspace of $\End_{\F_{q^2}}(\F_{q^{2n}})$. We explicitly note that if $f(x)=\sum_{i=1}^{n-1} c_i x^{q^{2i}}  \in H_{n}(q^2)$ with $n$ odd, then $c_{(n+1)/2} \in \F_{q^n}$. 

For given $a \in \F^*_{q}$, $\rho \in \Aut(\F_{q^2})$, $g$ a permutation $q^2$-polynomial over $\F_{q^{2n}}$, and $r_0 \in H_{n}(q^2)$, the map $\Theta: H_{n} (q^2)\rightarrow H_{n}(q^2)$ defined by 
\begin{equation}\label{eq:def_equivalence_hermitian_setting}
	\Theta_{a,g,\rho,r_0}(f) = ag \circ f^{\rho} \circ g^{{\top}q^{2n-1}}(x) + r_0(x), 
\end{equation} preserves the rank distance. The converse statement is also true, see \cite{wan_geometry_1996}  .

In this context if for $\cC_1,$ and $\cC_2 \in H_n(q^2)$, there exists a map $\Theta_{a,g,\rho,r_0}$ defined as in Equation \eqref{eq:def_equivalence_hermitian_setting} for certain $a$, $g$, $\rho$ and $r_0$ such that
\[ \cC_2 :=\{\Theta_{a,g,\rho,r_0}(f): f\in \cC_1 \},\]
then we say that $\cC_1$ and $\cC_2$ are \emph{equivalent} in $H_{n}(q^2)$, and write $\cC_1 \cong \cC_2$.

Regarding upper bounds for codes in this context, we may state the following result.
\begin{thm}\cite[Theorem 1]{schmidt_hermitian}\label{th:upper_bound_hermitian_association_scheme}
Assume that $\cC$ is an additive $d$-code in $H_n(q^2)$, then $$\# \cC \leq q^{n(n-d+1)}.$$
Moreover, when $d$ is odd, this upper bound also holds for non-additive $d$-codes.
\end{thm}

Let $s$ be an odd integer coprime with $n$. The following two classes of $\F_q$-linear codes in $H_n(q^2)$, were presented in \cite{schmidt_hermitian} only for $s=1$. However, the case with $s \in \mathbb{Z}, \,s \neq 1$ can be easily proved with the same technique used for generalizing Gabidulin codes in \cite{kshevetskiy_new_2005,Sheekey,lunardon_trombetti_zhou_2018}.

Suppose that $n$ and $d$ are integers with opposite parity such that $1 \leq d \leq n-1$. Then, the set
\begin{equation}\label{eq:hermitiancodeoppositeparity}
	\cH_{n,d,s}= \biggl \{\sum_{j=1}^{\frac{n-d+1}{2}} \biggl ( (b_jx)^{q^{2s(n-j+1)}}+b^{q^{s}}_jx^{q^{2sj}}  \biggr ) : \,\, b_1, b_2, \ldots, b_{\frac{n-d+1}{2}} \in \F_{q^{2n}} \biggr \},
\end{equation} 
is  a maximum $\F_q$-linear Hermitian $d$-code \cite[Theorem 4]{schmidt_hermitian}.

Also, suppose that $n$ and $d$ are both odd integers such that $1 \leq d \leq n$ and $s$ as above; then, the set 
\begin{equation}\label{eq:hermitiancodeodd}
\begin{aligned}
	\mathcal{E}_{n,d,s}=\biggl \{ (b_0\,x)^{q^{s(n+1)}}+ \sum_{j=1}^{\frac{n-d}{2}} \biggl ((b_jx)^{q^{s(n+2j+1)}}+b^{q^{s}}_jx^{q^{s(n-2j+1)}}  \biggr ) : \,\, b_0 \in \F_{q^n} \\ \,\,\text{and} \,\, b_1, \ldots, b_{(n-d)/2} \in \F_{q^{2n}}\biggr \}
\end{aligned}
\end{equation}
is a maximum $\F_q$-linear Hermitian $d$-code \cite[Theorem 5]{schmidt_hermitian}.

\section{Automorphism groups of known constructions}

Recall that the symbol $X_{n}$ denotes here one of the subspaces $S_{n}(q)$ and $A_n(q)$ of $\mathrm{End}_{\F_{q}}(\F_{q^n})$. Instead the symbol $H_n(q^2)$ is used to denote the $n^2$-dimensional $\F_q$-subspace of $\mathrm{End}_{\F_{q^2}}(\F_{q^{2n}})$ associated with a Hermitian form defined on $\F_{q^{2n}}$, with companion automorphism $a \mapsto a^q$. 

The aim here is determining the automorphism group of examples introduced in previous section. 

We start by giving an alternative description of such $d$-codes in terms of the intersection of their ambient space with suitable subspaces of $\mathcal{L}_{(n,q)}[x]$ (or of $\mathcal{L}_{(n,q^2)}[x]$, when dealing with the Hermitian setting). Precisely,

\begin{prop}\label{prop:d-codesproperty}
Let $n, s$ and $d$ be integers such that $1 \leq d \leq n$ and  $\gcd(s,n)=1$. Let ${\mathcal G} = \mathcal{G}_{n,n-d+1,s} \subset \mathcal{L}_{(n,q)}[x] $ be the generalized  Gabidulin code with minimum distance $d$, then we have the following
	
	\begin{enumerate}
		\item $\cS_{n,d,s}= \cG' \cap  S_n(q),$ where $\cG' =\mathcal{G} \circ x^{q^{s(\frac{n+d}{2})}}$.
		\item $\cA_{n,d,s}= \cG' \cap  A_n(q),$ where $\cG' =\mathcal{G} \circ x^{q^{s{\frac{d}{2}}}}$.
	\end{enumerate}
	
Moreover, let ${\mathcal G} = \mathcal{G}_{n,n-d+1,s} \subset \mathcal{L}_{(n,q^2)}[x] $ be the generalized  Gabidulin code with minimum distance $d$, then we have the following
	
	\begin{enumerate}
		\setcounter{enumi}{2}
		\item	$\cH_{n,d,s}= \cG' \cap  H_n(q^2),$ where $\cG' =\mathcal{G} \circ x^{q^{s(n+d+1)}}$.
		\item	$\mathcal{E}_{n,d,s}= \cG' \cap  H_n(q^2),$ where $\cG' =\mathcal{G} \circ x^{q^{s(d+1)}}$.
	\end{enumerate} 
\end{prop}
\proof Let $f (x)=  \sum_{i=0}^{n-d} a_i x^{q^{si}}$ be an element of $\mathcal{G}_{n,n-d+1,s}$. Each element in  $\mathcal{G}'=\mathcal{G}_{n,n-d+1,s} \circ x^{q^{s(\frac{n+d}{2})}}$ has the following form:
\begin{equation*}
  \sum_{i=0}^{n-d} a_i  x^{q^{s(\frac{n+d}{2}+i)}} = \,\, \sum_{i=0}^{\frac{n-d}{2}-1} a_i  x^{q^{s(\frac{n+d}{2}+i)}} +  \sum_{i=\frac{n-d}{2}}^{n-d} a_i  x^{q^{s(\frac{n+d}{2}+i)}} =
\end{equation*}
\begin{equation*}
\sum_{j=0}^{\frac{n-d}{2}} a_{j+\frac{n-d}{2}}  x^{q^{sj}} + \sum_{j=\frac{n+d}{2}}^{n-1} a_{j-\frac{n+d}{2}}  x^{q^{sj}} =
\end{equation*}

\begin{equation} 
a_{\frac{n-d}{2}}x + \sum_{j=1}^{\frac{n-d}{2}} \bigl ( a_{\frac{n-d}{2}+j}  x^{q^{sj}} + a_{\frac{n-d}{2}-j}  x^{q^{s(n-j)}} \bigr ).
\end{equation}

It is clear that $\cG'^{\top}=\cG'$, and by intersecting $\cG'$ with $S_n(q)$, we get the following conditions
\begin{equation*}
 a_{\frac{n-d}{2}-i}=a^{q^{s(n-i)}}_{\frac{n-d}{2}+i}, \,\,\,\,\,\,\,\, i=1,2,\ldots,\frac{n-d}{2}.
\end{equation*}
Hence, each element in $\cG' \cap S_n(q)$ has the following shape:
\begin{equation} \label{forma}
a_{\frac{n-d}{2}} x + \sum_{i=1}^{\frac{n-d}{2}} \bigl ( a_{\frac{n-d}{2}+i}  x^{q^{si}} + (a_{\frac{n-d}{2}+i}  x)^{q^{s(n-i)}} \bigr ). 
\end{equation}

This proves $(i)$. Point $(ii)$ is obtained arguing in the same way. 

Regarding $(iii)$ and $(iv)$, let $f (x)=  \sum_{i=0}^{n-d} a_i^{q^s} x^{q^{2si}}$ be any element in $\mathcal{G}_{n,n-d+1,s} \subset \mathcal{L}_{(n,q^2)}[x] $.  Composing $f(x)$ on the right with the monomial $x^{q^{s(n+d+1)}}$, we obtain
\begin{equation*}
a^{q^s}_0 x^{q^{s(n+d+1)}}+\sum_{i=1}^{\frac{n-d-1}{2}} a^{q^s}_i  x^{q^{2s \bigl (\frac{n+d+1}{2}+i \bigr)}} + \sum_{i=\frac{n-d+1}{2}}^{n-d} a^{q^s}_{i} \, x^{q^{ 2s \bigl (\frac{n+d+1}{2}+i\bigr )} }= 
\end{equation*}
\begin{equation*}
\sum_{j=0}^{\frac{n-d+1}{2}} a^{q^s}_{j+\frac{n-d-1}{2}} \,\, x^{q^{2sj}} + \sum_{j=\frac{n+d+1}{2}}^{n-1} a^{q^s}_{j-\frac{n+d+1}{2}} \, x^{q^{ 2sj}}= 
\end{equation*}
\begin{equation*}
\sum_{j=1}^{\frac{n-d+1}{2}}  \biggl ( a^{q^s}_{\frac{n-d+1}{2}-j} \,\, x^{q^{s(2n-2j+2)}} +a^{q^s}_{\frac{n-d-1}{2}+j} \, x^{q^{ 2sj}}\biggr ). 
\end{equation*}
By intersecting $\cG'$ with $H_n(q^2)$, we get the following conditions 
\begin{equation*}
c^{q^{s(2n-2j+1)}}_{j}=a^{q^{s(2n-2j+2)}}_{\frac{n-d-1}{2}+j}=c_{n-j+1}=a^{q^s}_{\frac{n-d+1}{2}-j}, \qquad \textnormal{for} \,\,\,j=1,2,\ldots,\frac{n-d+1}{2}.
\end{equation*}
Hence, we get 
\begin{equation*}
\cG' \cap H_n(q^2) = \cH_{n,d,s}. 
\end{equation*}

In a similar way, by composing an element $f(x) \in \mathcal{G}_{n,n-d+1,s}$ with $x \mapsto x^{q^{s(d+1)}}$, we obtain
\begin{equation*}
\sum_{i=0}^{n-d} a^{q^s}_i  x^{q^{2s \bigl (\frac{d+1}{2}+i \bigr)}}=
a^{q^s}_{\frac{n-d}{2}} \, x^{q^{s(n+1)}}+\sum_{i=0}^{\frac{n-d}{2}-1} a^{q^s}_{i} \,\, x^{q^{s(2i+d+1)}} + \sum_{i=\frac{n-d}{2}+1}^{n-d} a^{q^s}_{i} \, x^{q^{ s(2i+d+1)}}= 
\end{equation*}
\begin{equation*}
a^{q^s}_{\frac{n-d}{2}} \, x^{q^{s(n+1)}}+\sum_{i=1}^{\frac{n-d}{2}} a^{q^s}_{i-1} \,\, x^{q^{s(2i+d-1)}} + \sum_{j=\frac{n-d}{2}+1}^{n-d} a^{q^s}_{j} \, x^{q^{ s(2j+d+1)}}.
\end{equation*}
Setting $i=\frac{n-d}{2}-l+1$ and $j=\frac{n-d}{2}+m$, we obtain

\begin{equation*}
a^{q^s}_{\frac{n-d}{2}} \, x^{q^{s(n+1)}}+\sum_{l=1}^{\frac{n-d}{2}} a^{q^s}_{\frac{n-d}{2}-l} \,\, x^{q^{s(n-2l+1)}} + \sum_{m=1}^{\frac{n-d}{2}} a^{q^s}_{\frac{n-d}{2}+m} \, x^{q^{s( n+2m+1)}}=
\end{equation*}

\begin{equation*}
a^{q^s}_{\frac{n-d}{2}} \, x^{q^{s(n+1)}}+\sum_{j=1}^{\frac{n-d}{2}} \bigl ( a^{q^s}_{\frac{n-d}{2}-j} \,\, x^{q^{n-2j+1}} + a^{q^s}_{\frac{n-d}{2}+j} \, x^{q^{ s(n+2j+1)}} \bigr ).
\end{equation*}
Again by intersecting $\cG'$ with the Hermitian space $H_n(q^2)$, we get:
\begin{equation*}
	\begin{cases}
	a^{q^s}_{\frac{n-d}{2}} \in \F_{q^n} \\
	\\
	c_{\frac{n+1}{2}+j}^{q^{s(2n-2j)}}=a_{\frac{n-d}{2}+j}^{q^{s(2n-2j+1)}}=c_{\frac{n+1}{2}-j}= a^{q^s}_{\frac{n-d}{2}-j},
	\end{cases}
\end{equation*}
which finally gives the result.
\endproof

Regarding the punctured set obtained from $\cS_{n+1,d+2,s}$,  we can consider $\F_{q^{n+1}}\simeq\V \oplus \K$, where $\K=\langle \eta \rangle_q$ with $\eta \in \F^*_{q^{n+1}}$ and $\V$ is an $n$-dimensional $\F_q$-subspace of $\F_{q^{n+1}}$.\\
Let $s$ be a positive integer coprime with $n+1$, let $1 \leq d \leq n-1$ such that $n-d$ is odd, and consider the $\F_q$-vector space $\cU'_\eta$ of $\cG'=\cG_{n+1,n-d+2,s} \circ x^{q^{s \bigl (\frac{n+d+1}{2} \bigr )}}$ defined as follows 
\begin{equation} \label{subspace}
\begin{aligned}
\cU'_\eta=\Biggl \{ 
\sum_{i=1}^{\frac{n-d+1}{2}} \bigl ( c_{i}(x^{q^{si}}-x\eta^{q^{si}-1}) + 
 c_{n+1-i}(  x^{q^{s(n+1-i)}}-x\eta^{q^{s(n+1-i)}-1}) \bigr ) \\
 \,  :
  \,\, c_i,c_{n+1-i}\in \F_{q^{n+1}}, \,  \, i \in \biggl \{1,2,\ldots, {\frac{n-d+1}{2}} \biggl\}  \,\,\,  \Biggl \}.
\end{aligned}
\end{equation}

We notice that $\cU'_\eta$ has dimension $(n+1)(n-d+1)$, and it is made up of all maps $f\in \cG'$ such that $\K \subseteq \mathrm{Ker} f$. Let 
\begin{equation*}
\begin{aligned}
\cS_{n+1,d,s} \cap\, \cU'_\eta = \Biggl \{   \sum_{i=1}^{{\frac{n-d+1}{2}}}   \Bigl ( b_i (x^{q^{si}}-x\eta^{q^{si}-1})+b^{q^{s(n+1-i)}}_i(x^{q^{s(n+1-i)}}-x\eta^{q^{s(n+1-i)}-1})  \Bigr ) \\ \, : \,\,  b_1, \ldots, b_{{\frac{n-d+1}{2}}} \in \F_{q^{n+1}}  \Biggr \}.
\end{aligned}
\end{equation*}

Clearly each polynomial $f$ in this set has at most $q^{n-d+1}$ roots in $\F_{q^{n+1}}$. Furthermore, since $f$ is a linearized polynomial, we can write $f(x+u)=f(x)+f(u)$ for all $x,u \in \F_{q^{n+1}}$. But  $\K \subseteq  \mathrm{Ker}f$ which implies that, if $f(x)=0$, then $f(x+u)=0$ for all $u \in \K$. For each $x \in \V$ and each $u \in \K^*$, we have $x+u \not\in \V$, so the number of roots of the polynomial $f$ in $\V$ is at most $q^{n-d}$, i.e.
\begin{equation*}
\dim (\mathrm{Ker}\, f \cap \V) \leq n-d.
\end{equation*}
Hence, for each $f \in \cS_{n+1,d,s} \cap \cU'_{\eta}$, the rank of the symmetric bilinear form on $\V$ 
\begin{equation*}
{B^f}_{\vert_{\V}}:  (x,y) \in \V \times \V \rightarrow  \Tr_{q^n/q}(f(x)y)
\end{equation*}
 is at least $d$ and the set 
\begin{equation*}
\cT_{n,d,s}(\eta)=  (\cS_{n+1,d,s} \cap \cU_\eta' )_{\vert_{\V}}=\bigl \{ B{^f}_{\vert_{\V}} \,:\,f \in  \cS_{n+1,d,s} \cap  \cU_\eta' \bigr \}
\end{equation*}
is a symmetric  $\F_q$-linear maximum $d$-set of size $q^{(n+1){\frac{n-d+1}{2}}}$.\\

By Proposition \ref{prop:d-codesproperty} (i), we have the following.
\begin{cor}
Let  $(n+1,s)=1$, and $1 \leq d \leq n-1$.  Let  $\eta \in \F^\ast_{q^{n+1}}$ and let $\V$ be an $n$-dimensional $\F_q$-subspace of $\F_{q^{n+1}}$ such that $\F_{q^{n+1}}=\V \oplus \langle \eta \rangle_q$. Then the $d$-code
\begin{equation}
\mathcal{T}_{n,d,s}(\eta)= (\cU_\eta'  \cap S_{n+1}(q)  )_{|_{\V}},
\end{equation}
is maximum, where $\cU'_\eta$ is the $\F_q$-subspace in (\ref{subspace}).
\end{cor}  

Clearly, if $\eta_1$ and $\eta_2$ are linearly dependent over $\F_q$, then $\mathcal{T}_{n,d,s}(\eta_1)=\mathcal{T}_{n,d,s}(\eta_2)$.
Furthermore, we notice that $\cU'_\eta \, \cap \, S_{n+1}(q) \subset \cG' \cap S_{n+1}(q)=\cS_{n+1,d+2,s}$, while
\begin{equation}
\mathcal{T}_{n,d,s}(\eta)= (\cU_\eta'  \cap S_{n+1}(q))_{|_{\V}}= (\cS_{n+1,d+2,s})_{\vert_{\V}}.
\end{equation}

In the rest part of this section we prove that the subspace $\cG' \subset \mathcal{L}_{(n,q)}[x]$ ( $\cG' \subset \mathcal{L}_{(n,q^2)}[x]$) defined in Proposition \ref{prop:d-codesproperty}, is the unique element in $[\cG_{n,n-d+1,s}]_{\simeq}$ satisfying properties $(i)$ and $(ii)$ of Proposition \ref{prop:d-codesproperty}. More precisely, we have the following

\begin{thm}\label{thm:uniquesubspace}
Let $n, s$ and $d$ be integers such that $d \geq 1$ and $(s,n)=1$. 
	\begin{itemize}
		\item[(i)] Let $W \subset \mathcal{L}_{(n,q)}[x]$ be an $(n-d+1)n$-dimensional subspace of $\mathcal{L}_{(n,q)}[x]$ such that $W \in [\cG_{n,n-d+1,s}]_{\simeq}$, and $W \cap S_n(q)= \cS_{n,d,s}$  (respectively, $W \cap A_n(q)= \cA_{n,d,s}$).  
	
		Then $W=\cG_{n,n-d+1,s} \circ x^{q^{s\frac{n+d}{2}}}$  (respectively, $W =\cG_{n,n-d+1,s} \circ x^{q^{s{\frac{d}{2}}}}$).\\
	
		\item[(ii)]Let $W \subset \mathcal{L}_{(n,q^2)}[x]$ be an $(n-d+1)n$-dimensional $\F_{q^2}$-subspace such that $W \in [\cG_{n,n-d+1,s}]_{\simeq}$ and $W \cap H_n(q^2)= \cH_{n,t,s}$ (respectively, $W \cap H_n(q^2)= \mathcal{E}_{n,d,s}$). 
	
		Then, $W=\cG_{n,n-d+1,s} \circ x^{q^{s(n+d+1)}}$ (respectively, $W =\mathcal{G}_{n,n-d+1,s}\circ x^{q^{s(d+1)}}$).
	\end{itemize}
\end{thm}
\proof  $(i)$
Since $W$ is equivalent to $\cG_{n,n-d+1,s}$, there exists a rank-preserving map $\Phi_{g,\rho,h}$ such that 
\[ \Phi_{g,\rho,h}(\cG_{n,n-d+1,s})=W. \]
As $\cG_{n,n-d+1,s}^\rho=\cG_{n,n-d+1,s}$ for all $\rho \in \Aut( \F_q)$, we may assume that $\rho$ is the identity. Hence, the elements of $W$ are 
\begin{equation*}
g \circ \Biggr ( \sum_{j=0}^{n-d} \alpha_j x^{q^{sj}} \Biggr ) \circ h=  \sum_{j=0}^{n-d} (g \circ \alpha_j x^{q^{sj}} \circ h)=\sum_{j=0}^{n-d} \Biggl ( \sum_{m=0}^{n-1} c_{m,j}(\alpha_j)x^{q^{sm}} \Biggr )=
\end{equation*}
\begin{equation*}
\sum_{m=0}^{n-1} \Biggl ( \sum_{j=0}^{n-d} c_{m,j}(\alpha_j) \Biggr ) x^{q^{sm}},
\end{equation*}
with $\alpha_j \in \F_{q^n}$ for all $j \in J=\{0,1,\ldots,n-d\}$ and
\begin{equation*}
c_{m,j}(\alpha_j)=\sum_{i=0}^{n-1} g_i h^{q^{si}}_{m-i-j} \alpha_j^{q^{si}}.
\end{equation*}
The indices here are taken modulo $n$.

Suppose that $$W \cap S_n(q)= \cS_{n,d,s}.$$ 

By (\ref{symplectic space}) and (\ref{Schimdtcode}), we have that $L_m(\underline{\alpha})=\sum_{j=0}^{n-d} c_{m,j}(\alpha_j)$
is equal to zero for each $\underline{\alpha}=(\alpha_0,\alpha_1,\dots,\alpha_{n-d})$, $m \in M=\{\frac{n-d}{2}+1,\frac{n-d}{2}+2,\ldots, n-(\frac{n-d}{2}+1)\}$.

In particular $L_m(\underline{\alpha})=0$ when $\underline{\alpha}=(0,\ldots,0,\alpha_j,0,\ldots,0)$, with $\alpha_j \in \F_{q^n}$, $m \in M$ and $j \in J$. Then
\begin{equation*}
c_{m,j}(\alpha)=0 \,\,\,\hspace{0.5cm} \text{for all}\hspace{0.2cm}\alpha \in \F_{q^n}\,\,\text{and}\,\,\,m \in M, j \in J.
\end{equation*}
Hence, we obtain the following conditions:
\begin{equation}\label{system:conditions}
	\begin{cases}
	g_i h^{q^{si}}_{m-i-j}=0 \\
	i \in I:=\{0,1,...,n-1\}, \,j \in J, \,m \in M.
	\end{cases}
\end{equation}

As $g$ is an invertible $q$-polynomial, there exists at least an integer $i_0 \in I$ such that $g_{i_0} \neq 0$. It is straightforward to verify that 
$$\left\{m-j+\frac{n-d}{2} \,:\, j \in J \,\,\, \text{and} \,\,\, m \in M\right\}=\left\{1,2,...,n-1\right\}.$$ 
Hence, we get that for each given $i \in I$, by letting $j$ varying in $J$, the element $m-i-j$ may equal, modulo $n$, all elements  in $I$  with the only exception of $\frac{n+d}{2}-i$. But this finally implies that there exists a unique index $i_0$ between $0$ and $n-1$, such that $g_{i_0} \neq 0$ and $h_{\frac{n+d}{2}-i_0} \neq 0$; while all others $g_i$ and $h_i$ are zero. 

Hence, $g(x)=\gamma x^{q^{si_0}}$ and $h(x)=\delta x^{q^{s(\frac{n+d}{2}-i_0)}}$ with $\gamma,\delta \in \F_{q^n}$.

On the other hand if $$W \cap A_n(q)= \cA_{n,d,s},$$
by (\ref{alternatingspace}) and taking into account (\ref{eq:alternatingcode}), we may conclude that $$L_m(\underline{\alpha})=\sum_{j=0}^{n-d} c_{m,j}(\alpha_j)$$ is equal to zero for each $\underline{\alpha}=(\alpha_0,\alpha_1,\ldots,\alpha_{n-d}) \in \F_{q^n}^{n-d+1}$, $m \in M=M_1 \cup M_2 =\{0,1,\ldots,\frac{d}{2}-1\} \cup \{ n-(\frac{d}{2}-1),\ldots,n-1\}$. In particular $L_m(\underline{\alpha})=0$ for all $(0,\dots,\alpha_{j},\dots,0)$, with $\alpha_j \in \F_{q^n}$, $m \in M$.\\
Then
\begin{equation*}
	c_{m,j}(\alpha)=0 \,\,\,\hspace{0.5cm} \text{for all}\hspace{0.2cm}\alpha \in \F_{q^n}\,\,\text{and}\,\,\,m \in M, j \in J.
\end{equation*}
Hence, we obtain an analogous set of conditions; i.e.,
\begin{equation*}
	\begin{cases}
		g_i h^{q^{si}}_{m-i-j}=0 \\
		i \in I,j \in J, m \in M.
	\end{cases}
\end{equation*}

As $g$ is an invertible $q$-polynomial, there exists $i_0 \in I$ such that $g_{i_0} \neq 0$, and since $\frac{d}{2} \leq m-j-\frac{d}{2} \leq n-1$ for all $j \in J$. Again, one easily verifies that $$\left\{m-j-\frac{d}{2} \,\,:\,\, j \in J \,\,\, \text{and} \,\,\, m \in M_1 \cup M_2\right\}=\{1,2,...,n-1\}.$$

Again for each given $i \in I$, by letting $j$ varying in $J$, the element $m-i-j$ may be equal, modulo $n$ to all elements  of $I$, except $\frac{d}{2}-i$. Arguing as in the previous part this leads to prove that there exists a unique index $i_0$ between $0$ and $n-1$, such that $g_{i_0} \neq 0$ and $h_{\frac{d}{2}-i_0} \neq 0$; while all others $g_i$ and $h_i$ are zero.

Hence we have that $g(x)=\gamma x^{q^{si_0}}$ and $h(x)=\delta x^{q^{s(i_0+\frac{d}{2})}}$ with $\gamma,\delta \in \F^*_{q^n}$. This conclude the proof.

\medskip
\noindent $(ii)$ \,\, The proof of this point is similar to that of previous one. For this reason we omit here computations.
\endproof

As a direct consequence of Theorem \ref{thm:uniquesubspace}, we may state the following result.
\begin{cor}\label{cor:aut_groups}
Let $d$ and $s$ be integers such that $1<d<n$ and $\gcd(n,s)=1$. Let $\cC \in  X_n$ be a $d$-code.
	
\begin{itemize}
\item[(i)] If either $\cC= \cS_{n,d,s}$ or\, $\cC= \cA_{n,d,s}$. Then we have $$\Aut(\cC)=\left\{\Psi_{a, \gamma x^{q^r}} \,: \, a \in \F^\ast_q, \, \gamma \in \F^\ast_{q^n}, \,\, r \in \{0,...,n-1\} \right\}.$$  

\item[(ii)] If $\cC \in H_n(q^2)$ and either $\cC= \cH_{n,d,s}$ or\, $\cC= {\mathcal E}_{n,d,s}$. Then we have 
$$\Aut(\cC)=\left\{\Theta_{ a, \gamma^q x^{q^{2r}}} :\,\, a \in \F^\ast_{q}, \, \gamma \in \F^\ast_{q^{2n}} , \,\, r \in \{0,...,n-1\} \right\}.$$ 
\end{itemize}
\end{cor}	
\begin{proof}
$(i)$ We first observe that $\mathrm{Aut}(\cG')=\mathrm{Aut}(\cG_{n,n-d+1,s})$, whenever $\cG'=\mathcal{G} \circ x^{q^{s(\frac{n+d}{2})}}$ or  $\cG'=\mathcal{G} \circ x^{q^{s({\frac{d}{2})}}}$. Nonetheless, in \cite{Sheekey} it was proven that if $0 \leq r \leq n-1$, then $$\mathrm{Aut}(\cG_{n,n-d+1,s})=\bigl \{ \Phi_{\alpha x^{q^r},id,\beta x^{q^{n-r}}} \,|\, \alpha,\beta \in \F^\ast_{q^n} \bigr \}.$$ Now, assume that either $\cC=\cS_{n,d,s}$ or\, $\cC= \cA_{n,d,s}$. 

Since each element in the set
\begin{equation}\label{auto}
A=\{\Phi_{ a \gamma x^{q^r},id, \gamma^{q^{n-r}} x^{q^{n-r}}} |\,\, a \in \F^\ast_q, \, \gamma \in \F^\ast_{q^n}\}
\end{equation}
fixes both $S_n(q)$ and $A_n(q)$, then as a consequence of Proposition \ref{prop:d-codesproperty}, we get that $A$ is a subgroup of $\mathrm{Aut}(\cC)$. Conversely, let  $\Phi \in \mathrm{Aut}(\cC)$. Of course, by points $(i)$ and $(ii)$ of Proposition (\ref{prop:d-codesproperty}), we get $$\Phi(\cG') \cap \Phi(X_{n})=\cC,$$ whenever $X_n  = S_n(q)$ and  $\cG'=\mathcal{G} \circ x^{q^{s(\frac{n+d}{2})}}$, or $X_n  = A_n(q)$ and $\cG'=\mathcal{G} \circ x^{q^{s({\frac{d}{2})}}}$, respectively. This also means that $\cD = \Phi(\cG') \cap X_{n} \supseteq \cC.$ 

Now, assume that $\cD \supset \cC$. Then $\cD$ would be a $d$-code in $X_n$ with $\# \cD > \frac{n(n-d+2)}{2}$, which, since $n-d$ is even, by Theorem \ref{th:upper_bound} is clearly not possible. Hence, \begin{equation}\label{eq:intersectionproperty} \Phi(\cG') \cap X_{n} = \cC.\end{equation}  However, above Equation (\ref{eq:intersectionproperty}) contradicts Theorem \ref{thm:uniquesubspace}, unless we have $\Phi(\cG')=\cG'$, which implies that $\Phi$ is an element of $A$. This conclude the proof of point $(i)$.
	
	\medskip
	\noindent  $(ii)$\, Assume now that $\cC \subset H_n(q^2)$ is either $\cH_{n,d,s}$ or $\mathcal{E}_{n,d,s}$. Again, it is  trivial to see that, in both cases, each element in the set $$A=\{\Phi_{ a \gamma^q x^{q^{2r}},id,\gamma^{q^{2n-2r+1}} x^{q^{2n-2r}}} :\,\, a \in \F^*_{q}, \, \gamma \in \F^\ast_{q^{2n}}\},$$ fixes $\cC$. Moreover, an easy computation also shows that if either $\cC=\cH_{n,d,s}$ and $\cG' =\mathcal{G} \circ x^{q^{s(n+d+1)}}$, or $\cC=\mathcal{E}_{n,d,s}$ and  $\cG' =\mathcal{G} \circ x^{q^{s(d+1)}}$; we have

	\begin{equation}\label{eq:groupintersectioncondition}
		\Aut(\cG') \cap \Aut(\cC) = A.
	\end{equation} 

Now, let $\Phi=\Phi_{f,\rho,g}$ where $f$ and $g$ are two invertible $q^2$-polynomials in ${\mathcal L}_{(n,q^2)}[x]$, and $\rho \in \Aut(\F_{q^2})$, be an element of $\Aut(\cC)$, and suppose that $\Phi$ does not belong to  $A$. Then  by (\ref{eq:groupintersectioncondition}), $\Phi(\cG') \neq \cG'$ and this leads again to a contradiction by Theorem \ref{thm:uniquesubspace}.
\end{proof}

We end this section by proving the following equivalence results.

\begin{thm}\label{th:inequivalence_Cs}
Let $d \geq1$. Two maximum $d$-codes ${\mathcal S}_{n,d,s}$ and ${\mathcal S}_{n,d,s'}$ (respectively, ${\mathcal A}_{n,d,s}$ and ${\mathcal A}_{n,d,s'}$), where  $s$ and $s'$ are integers satisfying $\gcd(s,n)=\gcd(s',n)=1$, or, two maximal $d$-codes ${\mathcal H}_{n,d,s}$ and ${\mathcal H}_{n,d,s'}$ (respectively, $\mathcal{E}_{n,d,s}$ and $\mathcal{E}_{n,d,s'}$), where $s$ and $s'$ are  integers satisfying $\gcd(s,2n)=\gcd(s',2n)=1$,  are equivalent if and only if $s \equiv \pm s' \pmod{n}$.\\
\end{thm}
\begin{proof}
We give the proof only in symmetric and alternating setting. Similar arguments leads to the result for the two known constructions in the hermitian setting. For this reason we omit here the details.

Suppose that $s \equiv \pm s' \pmod{n}$. Let $\cG'_s =\mathcal{G}_s \circ x^{q^{s(\frac{n+d}{2})}}$ and $\cG'_{s'}=\mathcal{G}_{s'} \circ x^{q^{s'(\frac{n+d}{2})}}$ (respectively, $\cG'_s =\mathcal{G}_s \circ x^{q^{s({\frac{d}{2})}}}$ and $\cG'_{s'} = \mathcal{G}_{s'} \circ x^{q^{s'({\frac{d}{2})}}}$). 

By Proposition \ref{prop:d-codesproperty} points $(i)$ and $(ii)$, $${\mathcal S}_{n,d,s}=\cG'_s \cap S_n(q) \,\,\text{and}\,\, {\mathcal S}_{n,d,s'}=\cG'_{s'} \cap S_n(q),$$ (respectively, ${\mathcal A}_{n,d,s}=\cG'_s \cap A_n(q) \,\,\text{and}\,\, {\mathcal A}_{n,d,s'}=\cG'_{s'} \cap A_n(q)$).

Since $s \equiv \pm s' \pmod{n}$, by \cite[Theorem 4.4 and 4.8, $(a)$]{lunardon_trombetti_zhou_2018}, we have that $$\mathcal{G}_{s'}=\Phi_{ux^{q^r},id,vx^{q^{n-r}}}(\cG_s)=ux^{q^r} \circ \cG_s \circ vx^{q^{n-r}},$$ for two given elements $u,v\in \F_{q^n}$. Hence, 

$$\cG'_{s'}=(ux^{q^r} \circ \cG_s \circ vx^{q^{n-r}})\circ x^{q^{(\pm s+kn)(\frac{n+d}{2})}},$$ (respectively, $\cG'_{s'}=(ux^{q^r} \circ \cG_s \circ vx^{q^{n-r}})\circ x^{q^{(\pm s+kn)(\frac{d}{2})}}$),

If  $s' \equiv  s \pmod{n}$, from equation above we get \begin{equation}\label{eq:equivalence_issue1}\cG'_{s'}=ux^{q^r} \circ \cG_{s}' \circ v^{q^{s  \frac{n-d}{2}}} x^{q^{n-r}},\end{equation} (respectively, $\cG'_{s'}=ux^{q^r} \circ \cG_{s}' \circ v^{q^{-s(\frac{d}{2})}} x^{q^{n-r}}$).  

If otherwise $s' \equiv  -s \pmod{n}$, we have \begin{equation}\label{eq:equivalence_issue2}\cG'_{s'}=ux^{q^r} \circ (\cG_{s} \circ x^{q^{-s(\frac{n+d}{2})}}) \circ v^{q^{s \frac{n+d}{2}}}x^{n-r}, \end{equation} (respectively, $\cG'_{s'}=ux^{q^r} \circ ( \cG_{s} \circ x^{q^{-s(\frac{d}{2})}}) \circ v^{q^{s  (\frac{d}{2})}} x^{q^{n-r}}$).

Since $\cG_{s'}^{'\top} =\cG'_{s'}$, by comparing coefficients in Equation (\ref{eq:equivalence_issue1}) we get that it must necessarily be $u=av'$ with $a\in \F_q^*$ and $v'=v^{q^{s  \frac{n-d}{2}}}$ (respectively, $u=av'$ with $a\in \F_q^*$ and $v'= v^{q^{-s(\frac{d}{2})}}$). 

In a similar way, by comparing coefficients in Equation (\ref{eq:equivalence_issue2}), we find $u=av'$ with $a\in \F_q^*$, where $v'=v^{q^{s  \frac{n+d}{2}}}$ (respectively, $u=av'$ with $a\in \F_q^*$ and $v'= v^{q^{-s(\frac{d}{2})}}$).

Hence, $$\Phi_{av' x^{q^r},id,v'x^{q^{n-r}}}(\mathcal{S}_{n,d,s})=\Psi_{a,v'x^{q^r}}(\mathcal{S}_{n,d,s})=\mathcal{S}_{n,d,s'},$$ (respectively, $\Phi_{av' x,\rho,v'x}({\mathcal A}_{n,d,s})={\mathcal A}_{n,d,s'}$), where $s' \equiv  s \pmod{n}$.

Conversely, suppose that ${\mathcal S}_{n,d,s}$ and ${\mathcal S}_{n,d,s'}$ (respectively, ${\mathcal A}_{n,d,s}$ and ${\mathcal A}_{n,d,s'}$)  are equivalent. Denote by  $\Psi=\Psi_{a,g,\rho}$ the map such that $\Psi({\mathcal S}_{n,d,s})={\mathcal S}_{n,d,s'}$ (respectively, $\Psi({\mathcal A}_{n,d,s})={\mathcal A}_{n,d,s'}$). 

As  $\gcd(s,n)=\gcd(s',n)=1$, we may assume that $s'\equiv es \pmod{n}$. In the remaining part of the proof we will write down computation only in the symmetric context. Similar arguments may be applied in the alternating case leading to the same achievement. 

Each element $f\in\ {\mathcal S}_{n,d,s}$ has the following shape: $$ f(x)=b_0 x+  \sum_{i=1}^{\frac{n-d}{2}}   \Bigl ( b_i x^{q^{si}}+(b_ix)^{q^{s(n-i)}} \Bigl ).$$  Let $g= \sum_{i=0}^{n-1}a_i x^{q^{si}}\in \F_{q^n}[x]$.

Arguing as in the proof of Theorem  \ref{thm:uniquesubspace} we have that each element in $\Psi({\mathcal S}_{n,d,s})$ can be written as follows

	\begin{equation}\label{eq:B(Lx,Ly)}
	 \sum_{k=0}^{n-1}\left(
		\sum_{i=0}^{n-1}\left( b_0^{q^{s(n-i)}}a_i^{q^{s(n-i)}} a_{k+i}^{q^{s(n-i)}}
		+  \sum_{r=1}^{\frac{n-d}{2}}
		\left( b_r^{q^{s(n-i-r)}}a_i^{q^{s(n-i)}} a_{k+i+r}^{q^{s(n-i-r)}}
		+b_r^{q^{s(n-i)}}a_i^{q^{s(n-i)}} a_{k+i-r}^{q^{s(r-i)}}
		\right) \right) \right) x^{q^{sk}}.
	\end{equation}

By comparing the coefficients of the term $x^{q^{ks}}$ in $\Psi({\mathcal S}_{n,d,s})$ and in ${\mathcal S}_{n,d,s'}$ we get 

\begin{equation}\label{eq:coefficient_ykx}
\begin{aligned}
		\sum_{i=0}^{n-1}\left(b_0^{q^{s(n-i)}}a_i^{q^{s(n-i)}} a_{k+i}^{q^{s(n-i)}}
		+ \sum_{r=1}^{\frac{n-d}{2}}
		\left(b_r^{q^{s(n-i-r)}}a_i^{q^{s(n-i)}} a_{k+i+r}^{q^{s(n-i-r)}}
		+b_r^{q^{s(n-i)}}a_i^{q^{s(n-i)}} a_{k+i-r}^{q^{s(r-i)}}\right)
		\right)=0,
\end{aligned}
	\end{equation}
for each $k\in\{je : \frac{n-d}{2}<j<\frac{n+d}{2} \}$ and all $\lambda=(b_0, \cdots, b_{\frac{n-d}{2}})\in \F_{q^n}^{\frac{n-d}{2}+1}$.

By taking $b_0\neq0$ and $b_j=0$ for $j\neq 0$, from above Equation \eqref{eq:coefficient_ykx} we have
	\begin{equation}\label{eq:aiak+i=0}
		a_i a_{k+i}=0
	\end{equation}
	for $i=0,1,\dots, n-1$. Similarly, for each $r\in\{1, \dots, \frac{n-d}{2}\}$, letting $b_r$ be the unique nonzero elements among all $b_j$, from \eqref{eq:coefficient_ykx} we can derive
	\[\sum_{i=0}^{n-1}\left(a_{i-r}^{q^{s(r-i)}}a_{k+i}^{q^{s(n-i)}} + a_i^{q^{s(n-i)}}a_{k+i-r}^{q^{s(r-i)}} \right)b_r^{q^{s(n-i)}} =0.\]
	As the above equation holds for any $b_r\in \F_{q^n}$, it implies
	\[		a_{i-r}^{q^{s(r-i)}}a_{k+i}^{q^{s(n-i)}} + a_i^{q^{s(n-i)}}a_{k+i-r}^{q^{s(r-i)}} =0\]
	for every $i$, which means
	\begin{equation}\label{eq:ai-rak+i}
		a_{i-r}^{q^{sr}}a_{k+i} + a_ia_{k+i-r}^{q^{sr}}=0.
	\end{equation}

Since $g$ is a permutation $q$-polynomial, there must be at least one coefficient $a_i$, $i \in \{0,\dots,n-1\}$ which is different from zero. Denote such a coefficient with $a_{i_0}$.

By letting $i=i_0$ in \eqref{eq:aiak+i=0}, we get \[a_{je+i_0}=0 \, \text{ for }\frac{n-d}{2}<j<\frac{n+d}{2}.\]

	By taking $i=i_0$ and $i=r+i_0$ in \eqref{eq:ai-rak+i} respectively, together with the above equation, we can derive
	\[a_{i_0+je-r}=a_{i_0+je+r}=0 \, \text{ for }\frac{n-d}{2}<j<\frac{n+d}{2} \,\,\text{ and } \,\, 1\le r\le \frac{n-d}{2}. \]
	Hence,
	\[a_{je+i+i_0}=0 \text{ for } \, \frac{n-d}{2}<j<\frac{n+d}{2} \, \text{ and } -\frac{n-d}{2}\le i\le \frac{n-d}{2}.\]
	As $a_{i_0}\neq 0$, the equation
	\[je+i \equiv 0 \pmod{n}\]
	should have no solution for $\frac{n-d}{2}<j<\frac{n+d}{2}$ and $-\frac{n-d}{2}\le i\le \frac{n-d}{2}$.
	As there are $d-1$ elements in $\{je \pmod{n}: \frac{n-d}{2}<j<\frac{n+d}{2}\}$ and $n-d+1$ elements in $\{ i: -\frac{n-d}{2}\le i \le \frac{n-d}{2}\}$, $a_{je+i+i_0}=0$ implies all $a_j=0$ for $j\neq i_0$.
	
	Thus $g(x)=a_{i_0}x^{q^{i_0}}$. However, if $e \not\equiv \pm 1 \pmod{n}$, i.e.\ $s\not\equiv \pm s' \pmod{n}$, by Corollary \ref{cor:aut_groups}. it is obvious that $\Psi_{a,a_{i_0}x^{q^{i_0}},\rho}({\mathcal S}_{n,d,s})$ is not in ${\mathcal S}_{n,d,s'}$. Therefore, we must have $s\equiv \pm s' \pmod{n}$.
\end{proof}

\section{A Characterization of known additive constructions}

In this section we show that the property stated in Proposition \ref{prop:d-codesproperty} characterizing the known examples of maximum $d$-codes in restricted setting, up to the equivalence relation which we have denominated with the symbol $\cong$. More precisely, we prove the following

\begin{thm}\label{thm:characterization}
Let $n, s$ be two integers such that $n\geq 4$ and $(s,n)=1$, let $d$ be an integer such that $1 \leq d \leq n-1$. Let  $\cD \subset X_n$ be a maximum $d$-code.

\medskip

$(i)$ If $X_n = S_n(q)$, then $\cD\in [\cS_{n,d,s}]_{\cong}$ if and only if there is a unique subspace $V$ of $\mathcal{L}_{(n,q)}[x]$, such that

\begin{itemize}
	\item [$(a)$] $V \in [\cG']_{\simeq}$ where $\cG' = \cG_{n,n-d+1,s} \circ x^{q^{s{\frac{n+d}{2}}}}$;
	\item[$(b)$] $V=V^{\top}$, where $V^{\top}=\{f^{\top} \,:\, f \in V\}$;
	\item [$(c)$] $V \cap S_n(q)=\cD$.
\end{itemize}  

\medskip

$(ii)$ If $X_n = A_n(q)$, then $\cD\in [\cA_{n,d,s}]_{\cong}$ if and only if there is a unique subspace $V$ of $\mathcal{L}_{(n,q)}[x]$, such that 

\begin{itemize}
\item [$(a)$] $V \in [\cG']_{\simeq}$ where $\cG' = \cG_{n,n-d+1,s} \circ x^{q^{s\frac{d}{2}}}$;
\item[$(b)$] $V=V^{\top}$, where $V^{\top}=\{f^{\top} \,:\, f \in V\};$
\item [$(c)$] $V \cap A_n(q)=\cD$.
\end{itemize}
\end{thm}
\begin{proof}
Let us prove the sufficiency first. Assume $\cD \in [\cC]_{\cong}$ where either $\cC$ is $\cS_{n,d,s}$ or $\cC$ is $\cA_{n,d,s}$. Hence, there exists a rank-preserving map of type $\Psi=\Psi_{a,g,\rho}$, with $a \in \F^{*}_{q}$, $\rho \in \Aut(\F_{q})$ and $g$ a permutation $q$-polynomial, such that $\Psi(\cC)=\cD$. 

Let  $V=\Phi_{ag,\rho,g^\top}(\cG')$, where $\cG'=\mathcal{G} \circ x^{q^{s(\frac{n+d}{2})}}$ if $X_n=S_n(q)$ and $\cC=\cS_{n,d,s}$, and  $\cG' =\mathcal{G} \circ x^{q^{s({\frac{d}{2})}}}$ if  $X_n=A_n(q)$ and $\cC=\cA_{n,d,s}$. 

In both cases it is easy to see that $\cG'^{\top}=\cG'$. Hence, $V$ satisfies the properties $(a)$ and $(b)$. Moreover, as $\Phi_{ag,\rho,g^\top}$ fixes $X_n$, applying $(i)$ of Proposition \ref{prop:d-codesproperty}, we obtain that
\begin{equation*}
V \cap X_n=\Phi_{ag,\rho,g^\top}(\cG')\cap X_n=\Psi( \cG' \cap X_n)=\Psi(\cC)=\cD,
\end{equation*}
Hence $V$ satisfies $(c)$. 

Next, let us show the uniqueness. To this aim suppose that $V$ and $V'$ are two subspaces of $\mathcal{L}_{(n,q)}[x]$ both satisfying conditions $(a)$, $(b)$ and $(c)$. 

In particular we have  that $$V\cap X_n=\cD=V' \cap X_n.$$ 

By hypothesis $\cD = \Psi(\cC)$ and $\Psi$ fixes $X_n$. This means that there is an elements in $[\cG_{n,n-d+1,s}]_{\simeq}$ different from $\cG'$, intersecting $X_n$ in $\cC$. Indeed, $\Phi_{ag,\rho,g^\top}^{-1}(V')$. This, by Theorem \ref{thm:uniquesubspace} $(i)$, is a contradiction.

Now, let us prove the necessity. By $(a)$, $\cG'$ and $V$ are equivalent, more precisely there exists a map $\Phi=\Phi_{g,\rho,h}$ such that $V=\Phi(\cG')$. Since again $\cG'^{\top}=\cG'$, by using condition $(b)$, we have
\begin{equation}\label{eq:from((a)and((b)}
\Phi^{\top}(\cG')=\Phi_{h^{\top}, \rho, g^{\top}}(\cG')=V.
\end{equation}
Now, from (\ref{eq:from((a)and((b)}) and taking into account that $V=\Phi(\cG')$, we get
\begin{equation*}
\Phi_{g^{-1} \circ h^{\top},\mathrm{id},g^{\top} \circ h^{-1}}(\cG')=\cG'.
\end{equation*}
Hence, by Theorem \cite[Theorem 4]{Sheekey}, we get
\begin{equation*}
g^{-1} \circ h^{\top} = \alpha x^{q^r} \hspace{0.2cm} \text{and} \hspace{0.2cm} g^{\top}\circ h^{-1} = \beta x^{q^{n-r}},
\end{equation*}
with $\alpha,\beta \in \F^\ast_{q^n}$. 

In particular,  $r \equiv 0 \,(mod\,n)$  and consequently $\beta = \alpha^{-1},$ $g= h^{\top} \circ \beta x,$ and $\Phi=\Phi_{h^{\top} \circ \beta x, \rho, h}$.

We show that $\Phi(\cC) \cap \cD$ contains at least one element which is different from the null map. In fact, by $(c)$, we have that
\begin{equation*}
\dim\, (\Phi(\cC) \cap \cD) \geq \dim \, \Phi(\cC) + \dim \, \cD - \dim V = n. 
\end{equation*}
Hence, let $f$ be an element of $\cC$ such that $\Phi(f) \in \cD$. Since  $\Phi(f) \in \cD \subset S_n(q)$, we have that $\Phi^{\top}(f)=\Phi(f)$. Consequently,
\begin{equation*}
{f}^\rho(\beta x)=\beta {f}^\rho(x)  \quad\text{ for each } x \in \F_{q^n}.
\end{equation*}
Hence $\beta \in \F_q$ and
\begin{equation*}
\cD=\Phi(\cG') \cap X_n= \Phi(\cG' \cap X_n)=\Psi_{\beta,h^{\top},\rho}(\cC).
\end{equation*}
Hence $\cD \in [\cC ]_{\cong}.$
\end{proof}

A similar result can be stated also for the two known constructions of maximum $d$-codes in $H_n(q^2)$. Following is the precise statement

\begin{thm}
Let $n, s$ be two integers such that $(s,2n)=1$, and let $d$ be an integer such that $d > 1$. Then we have the following

\begin{itemize}
\item[(i)] $\cC\in [{\mathcal H}_{n,d,s}]_{\cong}$ if and only if there is an unique subspace $V$ of $\mathcal{L}_{(n,q^2)}[x]$, such that 

\begin{itemize}
\item [(a)] $V \in [\cG']_{\simeq}$ where $\cG' = \cG_{n,n-d+1,s} \circ x^{q^{s(n+d+1)}}$;
\item[(b)] $V=\tilde{V}$, where $\tilde{V}=\{\tilde{f} \,:\, f \in V\};$
\item [(c)] $V \cap H_n(q^2)=\cC$.
\end{itemize}  

\medskip
\item[(ii)] $\cC\in [{\mathcal E}_{n,d,s}]_{\cong}$ if and only if there is an unique subspace $V$ of $\mathcal{L}_{(n,q^2)}[x]$, such that

\begin{itemize}
\item [(a)] $V \in [\cG']_{\simeq}$ where $\cG' = \cG_{n,n-d+1,s} \circ x^{q^{s(d+1)}}$;
\item [(b)] $V=\tilde{V}$, where $\tilde{V}=\{\tilde{f} \,:\, f \in V\}$;
\item [(c)] $V \cap H_n(q^2)=\cC$.
\end{itemize}  

\end{itemize}
\end{thm}
\begin{proof}
The proof is similar to that of previous Theorem \ref{thm:characterization}; in fact, by simply taking into account that in this case we have $\tilde{\cG'}=\cG'$, whenever  $\cG' = \cG_{n,n-d+1,s} \circ x^{q^{s(n+d+1)}}$ or $\cG' = \cG_{n,n-d+1,s} \circ x^{q^{s(d+1)}}$.
\end{proof}

\section{A new additive symmetric  $2$-code}
Let $q$ be an odd prime power, $m$ and $s$ two integers such that $m \geq 2$ and $\gcd(s,2m)$. Let $\eta$ be an element of $\F_{q^{2m}}$ such that $N_{q^{2m}/q}(\eta)=\eta^{\frac{q^{2m-1}-1}{q-1}},$ is not a square.

 The set
\[ 
\mathcal{D}_{k,s}(\eta)= \left\{a x +\sum^{k-1}_{j=1} c_i x^{q^{js}}+ \eta bx^{q^{ks}}
\,: \, c_1,\cdots, c_{k-1} \in \F_{q^{2m}} ,\,\, a, b \in \F_{q^{m}} \right\}
\] 
is a maximum rank distance code with minimum distance $d = 2m - k + 1$ \cite{trombetti_zhou}. 

Now, let us consider the following set of $q$-polynomials
\begin{align*}
\cS=  \Biggl \{ a_0 x+ \sum^{m-2}_{j=1} a_j x^{q^{sj} }+ \eta b x^{q^{s(m-1)}}+ax^{q^{sm}}+ \eta^{q^ {s(m+1)}} b^{q^s} x^{q^{s(m+1)}} \\ +\sum^{m-2}_{j=1}(a_j x)^{q^{s(2m-j)}} \,\, : \,\,   a_0,a_1,\ldots,a_{m-2} \in \F_{q^{2m}} 
\,\,\textrm{and} \,\, a,b \in {\F_{q^m}} \Biggr \}.
\end{align*}
It is straightforward to see that, if we set $\mathcal{D}'=\mathcal{D}_{2m-1,s}(\eta) \circ x^{q^{sm}}$ then 
\begin{equation*}
\cS=\mathcal{D}'\cap S_{2m}(q) \,\,\,\,\textrm{and}\,\,\,\, \cS=\mathcal{D'}^{\top} \cap S_{2m}(q).
\end{equation*}

In what follows we will show that any map in $\cS$ has rank  strictly greater than one. In fact, let $f_m:=f \circ x^{q^{sm}}$, where $f \in \cS$. Then the coefficients of terms $x$ and $x^{q^{s(2m-1)}}$' of $f_m$ are $c$ and $\eta b$, respectively. As a consequence of \cite{gow-quinlan} (see also \cite[Lemma 3]{Sheekey}), the rank of $f$ is then at least two. Hence, $\cS$ is a maximum $2$-code of $S_{2m}(q)$. 

\begin{thm}
The $2$-code $\cS \in S_{2m}(q)$ is not equivalent to $\cS_{2m,2,s}$.
\end{thm}

\begin{proof} Assume by way of contradiction that $\cS$ is equivalent to $\cS_{2m,2,s}$. Then there must be a map $\Psi=\Psi_{a, g^{\top},\rho}$ such that $\Psi(\cS)=\cS_{2m,2,s}$,  where $a \in \F_{q}$, $\rho \in \Aut(\F_q)$ and $g(x)=\sum_{i=0}^{2m-1} g_i x^{q^{is}}$ is a permutation $q$-polynomial with coefficients in $\F_{q^{2m}}$.
	
Consider $g^{\top} \circ \alpha^ \rho x \circ g$, where $\alpha \in  \F_{q^{2m}}$. By computation the coefficient of $x^{{q^{ms}}}$ is
\begin{equation}
a_{m}(\alpha)=\sum^{2m-1}_{i=0} g^{q^{si}}_{2m-i}\,\, g_{m-i}^{q^{si}} \alpha^{ \rho q^{si}}
\end{equation}
where indices are taken modulo $2m$. Since the coefficient of the term with $q$-degree $ms$ of $\cS_{2m,2,s}$ is zero, we obtain 
\begin{equation*}
 g_{2m-i}\,\, g_{m-i}=0 \,\,\,\, \text{for each} \,\,\,\, i=1,2,\ldots,m.
\end{equation*}

Without loss of generality, we can suppose that 
$$g_{2m-i}=0 \quad\text{ for }\quad i=1,2,\ldots,m.$$

Let $c \in \F_{q^m}$, in the same way the coefficient of degree $q^{ms}$ of the composition $g^\top\circ c^\rho x^{q^{ms}} \circ g$ is equal to

\begin{equation*}
a_m(c)=\sum^{2m-1}_{i=0} g^{q^{s(2m-i)}}_i g^{q^{s(m-i)}}_{i}c^{\rho q^{s(2m-i)}}=\sum^{m-1}_{i=0} g^{q^{s(2m-i)}}_i g^{q^{s(m-i)}}_{i}c^{\rho q^{s(m-i)}}.
\end{equation*}
Obviously, since $$(g^{q^{s(2m-i)}}_i g^{q^{s(m-i)}}_{i})^{q^m}=g^{q^{s(2m-i)}}_i g^{q^{s(m-i)}}_{i},$$ the polynomial above has coefficients in $\F_{q^m}$. On the other hand, as  the coefficient of the term with $q$-degree $ms$ in $\cS_{2m,2,s}$ is zero, $a_m(c)=0$ for all $c \in \F_{q^m}$. This implies that $g_i=0$ for $i=0,1,\cdots, m-1$. Therefore $g$ is the null polynomial which contradicts the permutation property of $g$.
\end{proof}

\section*{Acknowledgment}
This work is supported by the Research Project of MIUR (Italian Office for University and Research) ``Strutture geometriche, Combinatoria e loro Applicazioni" 2012.

\end{document}